\documentclass[a4paper]{amsart}
\usepackage[latin9]{inputenc}
\usepackage{a4}
\usepackage[all]{xy}
\usepackage{amsfonts}
\usepackage{amssymb}
\usepackage{amsmath}
\usepackage{amsthm}
\usepackage{changepage}
\usepackage{nomencl}
\usepackage{geometry}
\usepackage{tikz}
\usepackage{verbatim}
\usetikzlibrary{matrix}
\usepackage[toc,page]{appendix}
\usepackage[%
    textsize=scriptsize%
    ]{todonotes}

\begin{document}
\providecommand{\keywords}[1]{\textbf{\textit{Keywords: }} #1}
\newtheorem{thm}{Theorem}[section]
\newtheorem{lemma}[thm]{Lemma}
\newtheorem{prop}[thm]{Proposition}
\newtheorem{cor}[thm]{Corollary}
\theoremstyle{definition}
\newtheorem{defi}[thm]{Definition}
\theoremstyle{remark}
\newtheorem{remark}[thm]{Remark}
\newtheorem{prob}[thm]{Problem}
\newtheorem{conjecture}[thm]{Conjecture}
\newtheorem{ques}[thm]{Question}

\newcommand{\cc}{{\mathbb{C}}}   
\newcommand{\ff}{{\mathbb{F}}}  
\newcommand{\nn}{{\mathbb{N}}}   
\newcommand{\qq}{{\mathbb{Q}}}  
\newcommand{\rr}{{\mathbb{R}}}   
\newcommand{\zz}{{\mathbb{Z}}}  
\newcommand{\K}{\mathbb{K}}
\newcommand{\ra}{\rightarrow}
\newcommand{\Gal}{\textrm{Gal}}

\title{A note on invariable generation of nonsolvable permutation groups}
\author{Joachim K\"onig}
\address{Department of Mathematics Education, Korea National University of Education, Cheongju, South Korea}
\email{jkoenig@knue.ac.kr}
\author{Gicheol Shin}
\address{Department of Mathematics Education, Seowon University, Cheongju, South Korea}
\email{gshin@ucdavis.edu}
\begin{abstract}
We prove a result on the asymptotic proportion of randomly chosen pairs $(\sigma,\tau)$ of permutations in the symmetric group $S_n$ which ``invariably" generate a nonsolvable subgroup, i.e.,  
whose cycle structures cannot possibly both occur in the same solvable subgroup of $S_n$.
As an application, we obtain that for a large degree ``random" integer polynomial $f$, reduction modulo two different primes can be expected to suffice to prove the nonsolvability of $\Gal(f/\mathbb{Q})$.
\end{abstract}

\maketitle

\section{Introduction and main result}


This paper is motivated by the following, at this point deliberately vaguely worded question:

\begin{ques}
\label{ques:nonsolv}
Given a ``random" integer polynomial $f\in \mathbb{Z}[X]$, how ``difficult" is it to verify the nonsolvability of the Galois group $\Gal(f/\mathbb{\mathbb{Q}})$?
\end{ques}

To begin with, the precise meaning of ``random polynomial" is up to the choice of a concrete model. For our purposes, we essentially mean ``a degree $n$ polynomial whose Galois group is  $A_n$ or $S_n$" (although of course Question \ref{ques:nonsolv} is to be understood such that whoever aims at proving nonsolvability is provided with no prior information on the Galois group of the concrete polynomial), or slightly more general (and in order to allow a notion of randomness) ``a polynomial $f$ chosen from a family in such a way that $\Gal(f/\mathbb{Q})$ is almost certainly alternating or symmetric". The latter is known to hold for several versions of random polynomials.
For example, a famous result by van der Waerden (\cite{vdW}) asserts that for any fixed degree $n\in \mathbb{N}$ and for $N\to \infty$, the proportion of polynomials with Galois group $S_n$ among degree-$n$ integer polynomials with coefficients in $\{-N,\dots, N\}$ tends to $1$.
More recently, Bary-Soroker and Kozma showed in \cite{BSK} that if $N$ is divisible by four distinct primes, then for $n\to \infty$, the probability of a degree-$n$ polynomial with coefficients chosen uniformly at random from the set $\{1,\dots, N\}$ to have Galois group containing $A_n$ tends to $1$. This viewpoint (i.e., degree tending to infinity) will be more relevant for us in view of Theorem \ref{thm:main} below.

Next, regarding the difficulty of verification mentioned in Question \ref{ques:nonsolv}, we only consider one approach to obtain information about the Galois group of an integer polynomial, based on the following well-known criterion on modulo-$p$ reduction due to Dedekind.\footnote{While Dedekind's criterion of course in general cannot answer all questions about the Galois group, it is considerably less expensive than other approaches such as resolvent methods.}

\begin{thm}[Dedekind]
\label{thm:dedekind}
Let $f\in \mathbb{Z}[X]$ be a separable polynomial of degree $n$ with Galois group $G=\Gal(f/\mathbb{Q})\le S_n$, let and $p$ be a prime dividing neither the leading coefficient of $f$ nor the discriminant of $f$. Then the Galois group of the mod-$p$ reduction $\overline{f}$ of $f$ (acting on the roots of $\overline{f}$) embeds as a permutation group into $G$. In particular, if $d_1,\dots, d_r$ are the degrees of the irreducible factors of $\overline{f}$ over $\mathbb{F}_p$, then $G$ contains an element whose cycle lengths are exactly $d_1,\dots, d_r$.
\end{thm}

Conversely, given $f\in \mathbb{Z}[X]$ with Galois group $G\le S_n$, the existence of primes whose reduction yields any prescribed cycle structure inside $G$ is guaranteed and quantified by Frobenius' density theorem (later famously strengthened by Chebotarev).

\begin{thm}[Frobenius]
\label{thm:frob}
Let $f\in \mathbb{Z}[X]$ be a separable degree-$n$ polynomial with Galois group $G\le S_n$, and let $C:=[c_1,\dots, c_r]$ be a cycle type in $G$. Then the asymptotic density of primes $p$ for which mod-$p$ reduction of $f$ in Theorem \ref{thm:dedekind} yields cycle type $C$ equals the proportion of elements of cycle type $C$ in $G$. 
\end{thm}

As (a special case of) a well-known theorem of Jordan, a subgroup intersecting all conjugacy classes of $S_n$ nontrivially must be $S_n$ itself. Thus, Dedekind's and Frobenius' theorems guarantee that, in the case $\Gal(f/\mathbb{Q})=S_n$, the above reduction process will eventually correctly identify the Galois group, if sufficiently many primes are chosen. Efficient bounds on the number of primes required to succeed with certainty are hard to obtain (and depend of course on the concrete polynomial $f$), see, e.g., \cite{LMO}. We instead ask how many primes are needed to succeed ``with high probability". 

The following definition is directly motivated by Frobenius' theorem, which yields, for each mod-$p$ reduction of a given polynomial $f$, not a concrete element, but only a cycle type which is guaranteed to occur in the Galois group of $f$.

\begin{defi}
\label{def:inv}
Let $G\le S_n$, let $\mathcal{S}$ be a family of subgroups \underline{of $S_n$}, and $\sigma_1,\dots, \sigma_r\in G$. Say that $\sigma_1,\dots, \sigma_r$ invariably generate a member of $\mathcal{S}$, if for all $\tau_1,\dots, \tau_r\in S_n$ such that $\tau_i$ is conjugate \underline{in $S_n$} to $\sigma_i$ for all $i=1,\dots, r$, the group generated by $\tau_1,\dots, \tau_r$ belongs to $\mathcal{S}$.
\end{defi}

The dependency of the underlined parts in Definition \ref{def:inv} on $S_n$ rather than $G$ is deliberate. Of course, switching them from $S_n$ to $G$ would also be meaningful, and would indeed be common in a purely group theoretical context. E.g., with this notion of ``invariable generation", it is known that every non-abelian simple group $G$ has a pair of elements all of whose $G$-conjugates generate $G$, see \cite[Theorem 1.3]{KLS}, and that the expected number of $G$-conjugacy classes (of an arbitrary group $G$) required to invariably generate $G$ is in $O(\sqrt{|G|})$, see \cite{Luc}. 
However, our version seems natural in view of the Galois theoretical interpretation, since it corresponds to a situation where no prior information on the true Galois group $G\le S_n$ is available.

For the key case $G=S_n$, invariable generation of transitive subgroups  (i.e., $\mathcal{S}$ equal to the set of all transitive subgroups of $S_n$), or indeed of $S_n$ itself, has been studied by a number of authors. 
In particular, the following two key questions have been studied in detail: Firstly, what is the minimal number $r_0$ such that permutations $\sigma_1,\dots, \sigma_{r_0}\in S_n$, chosen independently and uniformly at random from the set $S_n$ (for short: random permutations), invariably generate a transitive subgroup (resp., invariably generate $S_n$) with probability bounded away from $0$ as $n\to \infty$? \footnote{Note that in this question one cannot hope for a probability tending to $1$ for any fixed $r_0$, since the proportion of elements of $S_n$ having a fixed point is bounded away from $0$.} 
 Secondly, what is the {\it average} number of random permutations needed to invariably generate a transitive subgroup (resp., $S_n$)? 
Dixon conjectured in \cite{Dixon} that 
the average number of random permutations required to invariably generate a transitive subgroup should be absolutely bounded from above 
as $n\to \infty$. This was first shown to be true in \cite{LP}.
Later, Pemantle, Peres and Rivin showed in \cite{PPR} that in fact four random permutations invariably generate a transitive subgroup, and even $S_n$, with 
probability bounded away from $0$. On the other hand, for three random permutations, the probability to invariably generate a transitive subgroup is known to tend to $0$ as $n\to \infty$, by \cite{EFG2}. 
All these results have natural interpretations in Galois theory via Dedekind's criterion; e.g., one may expect to obtain an irreducibility certificate for a large degree polynomial with Galois group $S_n$ via a bounded number of modulo-$p$ reductions; three reductions will almost certainly be insufficient, whereas four give a non-negligible chance to succeed. Finally, the aforementioned  \cite{BSK}  may be seen as a more advanced connection between invariable generation and Galois theory,  
 combining results in the spirit of \cite{PPR} with results on factorization of polynomials with restricted coefficients to obtain a ``$100\%$" statement (in the limit $n\to \infty$) on the Galois side. 
%

Motivated by Question \ref{ques:nonsolv}, we are interested in the case where $\mathcal{S}$ in Definition \ref{def:inv} is the set of all nonsolvable subgroups of $S_n$. Concretely, we prove the following.
\begin{thm}
\label{thm:main}
Let $G=S_n$ or $G=A_n$. Then the proportion of pairs $(\sigma, \tau)\in G^2$ which invariably generate a nonsolvable subgroup of $S_n$ is $1-O(1/(\log n)^{2-\epsilon})$, for every $\epsilon>0$.
In particular, two random permutations $\sigma,\tau\in G$ invariably generate a nonsolvable subgroup with probability tending to $1$ as $n\to \infty$.
\end{thm}

Regarding Question \ref{ques:nonsolv}, this implies that in order to find a certificate for the nonsolvability of $\Gal(f/\mathbb{Q})$ for a given large degree integer polynomial $f$ with Galois group $S_n$ or $A_n$, it almost certainly suffices to reduce modulo two different primes. Obviously this is best possible, since a single mod-$p$ reduction will only yield one particular cyclic subgroup of $\Gal(f/\mathbb{Q})$, and thus cannot even rule out this group to be cyclic.


The proof of Theorem \ref{thm:main}, contained in Section \ref{sec:proof}, makes decisive use of a recent combinatorial result by Unger (\cite{Unger}); in terms of permutation group theory, it requires only elementary prerequisites. 
In the appendix, we show how to obtain a more concise version of the nonsolvability conclusion of Theorem \ref{thm:main}, which requires some more advanced classification results about primitive groups.


\section{Prerequisites}

We include some basic combinatorial and group-theoretical facts about permutations which will be used later.

\subsection{Permutations with partially prescribed cycle structure}
Let $\lambda$ be a partition of $k$. For any permutation $\sigma \in S_k$ of cycle type $\lambda$, we denote by $\lambda!$ the size of the centralizer $C_{S_k}(\sigma)$ in $S_k$, i.e., $\lambda! = \prod_{i = 1}^{\infty} {i^{m_i(\lambda)}m_i(\lambda)!}$, where $m_i(\lambda)$ denotes the multiplicity of $i$ in $\lambda$.
Recall that the number of permutations of the cycle type corresponding to $\lambda$ is equal to $\frac{k!}{\lambda!}$.

Let $\lambda$ be a partition of $k$ for some $k \le n$. We denote by $N_{\lambda,n}$ the number of permutations in $S_n$ which contain at least $m_1(\lambda)$ $1$-cycles, $m_2(\lambda)$ $2$-cycles, $m_3(\lambda)$ $3$-cycles, and so on, i.e.,
$$
    N_{\lambda,n} = \#\{\sigma \in S_n : \text{the multiplicity of $i$ in the cycle type of $\sigma$ is $\ge m_i(\lambda)$ for each $i$}\}.
$$
To compute (an upper bound of) the number $N_{\lambda, n}$, we can pick a $k$-subset $X$ of $\{1, 2, \dots, n\}$ in $\binom{n}{k}$ ways, then pick a permutation on $X$ of cycle type $\lambda$ in $\frac{k!}{\lambda!}$ ways, and then pick any permutation on $\{1,2,\dots,n\}\setminus X$ in $(n-k)!$ ways. Thus, the total number of ways is equal to
$$
    \binom{n}{k} \cdot \frac{k!}{\lambda!} \cdot (n-k)! = \frac{n!}{\lambda!}.
$$
In the above we have counted permutations of cycle type $\mu$ exactly $\prod_{i=1}^{\infty} \binom{m_i(\mu)}{m_i(\lambda)}$ times, where $\mu$ is a partition of $n$ such that $m_i(\mu) \ge m_i(\lambda)$ for each $i$. Hence, we have $N_{\lambda, n} \le \frac{n!}{\lambda!}$. Consequently, we have:

\begin{lemma}
\label{lem:prob}
Let $m_1, \dots, m_n$ be non-negative integers.
The probability that a random permutation in $S_n$ contains at least $m_1$ 1-cycles, $m_2$ 2-cycles, $m_3$ 3-cycles, and so on, cannot exceed $\prod_{i=1}^{n}\dfrac{1}{i^{m_i} m_i !}$. In particular, the probability that a random permutation in $S_n$ contains at least one $k$-cycle is less than or equal to $\frac{1}{k}$.
\end{lemma}



\subsection{Primitive and imprimitive permutation groups}

A transitive permutation group $G\le S_n$ is called primitive if it does not preserve a non-trivial block system in $\{1,\dots, n\}$, or equivalently, if the point stabilizers are maximal subgroups in $G$. The following are classical results about the structure of primitive permutation groups. The first one is due to Jordan (see \cite[Theorem 3.3E]{DM}), the second one was essentially known to Galois (cf.\ \cite{Neumann}).

\begin{thm}[Jordan]
\label{thm:jordan}
Let $G$ be a primitive permutation group of degree $n$, containing a cycle of prime length fixing at least three points. Then $G\ge A_n$.
\end{thm}

\begin{thm}
\label{thm:affine}
Let $G\le S_n$ be a primitive and solvable group. Then $G$ acts as an affine group on some vector space. In particular, $n$ is a prime power. If furthermore $n=p$ is a prime, then $G$ is contained in the normalizer $N_{S_p}(\langle \sigma\rangle) \cong C_p\rtimes C_{p-1} =:AGL_1(p)$, for some $p$-cycle $\sigma$.
\end{thm}

Recall furthermore that if $G\le S_n$ acts transitively but imprimitively, then there exist $a,b>1$ such that $n=ab$ and transitive groups  $U\le S_a$, $V\le S_b$ such that $G$ embeds into the wreath product $U\wr V := U^b \rtimes V$ (with $V$ acting via permuting the $b$ copies of $U$). Furthermore, if $G_1<G$ denotes a point stabilizer, then there exists $G_1 < H_1 < G$ such that the image of $G$ in the action on cosets of $H_1$ equals $V$, and the image of $H_1$ in the action on cosets of $G_1$ in $H_1$ equals $U$. See, e.g., \cite{DM} for more background on primitive and imprimitive permutation groups.

\section{Proof of Theorem \ref{thm:main}}
\label{sec:proof}
\begin{proof}
Firstly, it obviously suffices to consider $G=S_n$. Since $A_n$ comprises half the permutations of $S_n$, the assertion for $A_n$ follows readily.
We pick two permutations $\sigma,\tau\in S_n$ independently at random. 
We will show that the following property of $\sigma, \tau$ holds with probability $1-O((\frac{\log\log n}{\log n})^2)$.
\begin{equation}
\label{eq:1} \text{For all } x\in S_n, \text{ the subgroup } \langle \sigma, \tau^x\rangle \le S_n \text{ is nonsolvable}.
\end{equation}

We will achieve this through a series of claims. Note that whenever we encounter a property fulfilled by a proportion of $O((\frac{\log\log n}{\log n})^2)$ pairs of permutations $(\sigma,\tau)\in S_n^2$ (or in particular, a property fulfilled by a proportion of $O((\frac{\log \log n}{\log n})^2)$ permutations $\sigma\in S_n$), we may exclude the permutations fulfilling this property from the further considerations, simply since the union of a bounded number of sets of a certain asymptotic $O(f(n))$ is still in $O(f(n))$.

{\it Step I}: We first claim that the following property holds for a proportion $1-O((\frac{\log\log n}{\log n})^2)$ of all pairs of permutations $(\sigma,\tau) \in S_n^2$:
\begin{equation}
\label{eq:2}
\begin{split}
\text{ There exists a prime } p\in [\log(n)^2, n/2]  \text{ such that at least one of } \sigma, \tau \\ \text{ contains a } p-\text{ cycle and has all of its other cycle lengths coprime to } p.
\end{split}
\end{equation}

Indeed, it suffices to show that a proportion $1-O(\frac{\log\log n}{\log n})$ of permutations $\sigma\in S_n$ fulfill the condition in \eqref{eq:2}. {\it Without} the upper bound condition $p\le n/2$, this is shown in Theorems 2 and 5 of \cite{Unger}. To see that the upper bound may be added without changing the estimate on the asymptotic proportion, note that due to Lemma \ref{lem:prob}, the probability for $\sigma$ to contain a $p$-cycle for some prime $p$ larger than $n/2$ is at most - and in fact in this special case, equal to (since the individual events are mutually exclusive) $\sum\limits_{\stackrel{p\in \mathbb{P}}{n/2< p\le n}} \frac{1}{p}$. Now it is well known that $(\sum\limits_{\stackrel{p\in \mathbb{P}}{p\le n}} \frac{1}{p}) - \log\log n$ converges to a finite value $M$ (namely, the so-called Meissel-Mertens constant) for $n\to \infty$. Therefore, as $n$ grows, $\sum\limits_{\stackrel{p\in \mathbb{P}}{n/2< p\le n}} \frac{1}{p}$ approaches $\log\log n - \log\log(\frac{n}{2}) = \log\log n - \log(\log n-\log(2))$. Since $\log(x) - \log(x-c)$ is of the order of growth of $\frac{c}{x}$ (for $x\to \infty$ and $c$ constant), we obtain in total that the probability that $\sigma$ contains a $p$-cycle for some prime $p>\frac{n}{2}$ is in $O(\frac{1}{\log n})$, and a fortiori in $O(\frac{\log\log n}{\log n})$. Thus, the requirement $p\le n/2$ in Condition \eqref{eq:2} may be added without changing the asymptotic.

From now on, we thus may and will assume that Condition \ref{eq:2} holds (for one of $\sigma$ and $\tau$, and for some prime $p\in [\log(n)^2, n/2]$). Up to renaming, we may and will assume from now on that $\sigma$ fulfills the condition, and will furthermore denote the $p$-cycle of $\sigma$ by $(1,2,\dots, p)$, which is of course also without loss of generality.

{\it Step II}: We next claim that the proportion of $\sigma\in S_n$ containing a cycle of length a Mersenne prime $\ge \log(n)^2$ is $O(\frac{1}{\log(n)^2})$, whence we may and will additionally assume in the following that the prime $p$ in Condition \eqref{eq:2} is {\it not} a Mersenne prime.

Indeed, we may again use Lemma \ref{lem:prob} together with a simple geometric series estimate, to bound the probability for $\sigma$ to contain a cycle of some length $2^k-1\ge \log(n)^2$ from above by $\sum\limits_{j=0}^\infty \frac{1}{\log(n)^2} \cdot (\frac{1}{2})^j$, which is in $O(\frac{1}{\log(n)^2})$. 

We now include the second permutation $\tau$ in our considerations, and (for all possible choices of $x\in S_n$) look at the length of the orbit of the subgroup $\langle \sigma, \tau^x\rangle$ containing the orbit $\{1,\dots, p\}$ of $\langle \sigma\rangle$. We denote this orbit by $O$ and from hereon consider the image $U$ of $\langle\sigma, \tau^x\rangle$ in its action on $O$. Of course it suffices to show that for a proportion $1-O((\frac{\log \log n}{\log n})^2)$ of pairs $(\sigma,\tau)$, there exists no $x$ such that $U$ is solvable.

{\it Step III}: We claim that, for any $\sigma$ as above, the probability (in $\tau$) that there exists at least one $x\in S_n$ for which the above $U\le Sym(O)$ is solvable and primitive is $O(\frac{1}{\log(n)^2})$.

Indeed, due to Step I, $U$ contains a $p$-cycle (namely, a suitable power of $\sigma$, restricted to $O$).  Thus, from Theorem \ref{thm:jordan}, if $U$ is primitive and solvable, then certainly $|O| \in \{p, p+1, p+2\}$. The case $|O| = p+2$ can also be excluded by elementary means.\footnote{Or indeed by non-elementary ones such as the classification result in \cite[Theorem 1.2]{Jones}.} Indeed, by Theorem \ref{thm:affine}, a primitive solvable group is affine, i.e., contained in $AGL(V) \cong V\rtimes GL(V)$ for some finite vector space $V$. The point stabilizers in $AGL(V)$ are conjugate to (the point stabilizer of the zero vector) $GL(V)$, so if $U\le AGL(V)$ contains a cycle fixing exactly two points, then in particular $GL(V)$ contains an element fixing exactly one non-zero vector. But the latter can only happen if $V$ is an $\mathbb{F}_2$-vector space, i.e., $|O| = p+2$ is a $2$-power, which is impossible.
Furthermore, the case $|O|=p+1$ can be excluded via Step II, since if $U$ were primitive and solvable (and thus, affine) of degree $p+1$, then $p+1$ would be a prime power, i.e., $p$ would be a Mersenne prime. So we are left with the case that $U$ is primitive and solvable of degree $p$, i.e., $U\le AGL_1(p)$. But the only elements of $AGL_1(p)$ are $p$-cycles and powers of a $(p-1)$-cycle, i.e., elements of cycle type $[1,d,\dots, d]$ for $d$ dividing $p-1$. We need to bound the proportion of $\tau\in S_n$ having such a pattern in their cycle structure. Of course, since $p\ge \log(n)^2$, the probability for $\tau$ to contain a $p$-cycle is $\le \frac{1}{\log(n)^2}$ by Lemma \ref{lem:prob}. On the other hand, via setting $a:=\frac{p-1}{d}$, the combined probability to contain $a$ cycles of length $d$ for any $d|p-1$ is bounded from above by $\sum\limits_{a | p-1} \frac{a^a}{(p-1)^a \cdot a!}$. We may now use the inequality $$a! \ge \sqrt{2\pi} a^{a+1/2}\exp(-a),$$
(a part of Stirling's formula valid for all $a\in \mathbb{N}$), to obtain an upper bound $\frac{1}{\sqrt{2\pi}} \sum\limits_{a | p-1} \frac{\exp(a)}{(p-1)^a \sqrt{a}} \le \sum_{a\ge 1} (\frac{e}{p-1})^a$. Due to $p\ge \log(n)^2$, the term for $a=1$ gives a contribution $O(\frac{1}{\log(n)^2})$, whereas the remaining sum is strictly smaller. This completes the proof of the claim of Step III.




We may therefore assume from now on that $\sigma$ and $\tau$ are such that, for any $x\in S_n$, the group $U$ above is either nonsolvable or imprimitive on $O$. 
To bound the proportion of pairs $(\sigma,\tau)$ for which $U$ is solvable and imprimitive for at least one $x\in S_n$, we return to the investigation of the cycle structure of $\sigma$.

 Let $\{O_1,\dots, O_k\}$ (with $\cup_{i=1}^k O_i = O$) be a block system of minimal block length preserved by  the corresponding group $U$.
%
%
Then $U\le G\wr H = (\underbrace{G\times \dots\times G}_{k \text{ times}})\rtimes H$, where $G\le Sym(O_1)$ is primitive (due to the minimality assumption), and $H\le S_k$ is transitive. Furthermore, recall that some power of $\sigma$ acts as a $p$-cycle on $O$, and this $p$-cycle necessarily lies in one component $G$ (i.e., it permutes one block while fixing all the others pointwise). We assume $U$ to be solvable, in which case $G$ is solvable (and in particular affine) as well. But $G$ also contains a $p$-cycle, so $p<|O_1|-2$ is excluded by Theorem \ref{thm:jordan}, and $|O_1|\in \{p+1, p+2\}$ can be excluded as well, as already shown in Step II.

We have thus obtained $|O_1|=p$, so $G\le AGL_1(p)$.
We choose the ordering of the blocks $O_i$ such that $O_1\ne \{1,\dots, p\}$
Assume without loss of generality that $\sigma$ permutes the $m$ blocks $O_1,\dots, O_m$ cyclically ($m\ge 1$), and denote the corresponding cycle lengths on the order $mp$ set $O_1\cup\dots \cup O_m$ by $md_1$, $\dots$, $md_r$ ($d_i\in \mathbb{N}$). Then $\sigma^m$ fixes the block $O_1$ setwise, and induces a permutation of cycle lengths $d_1,\dots, d_r$ on this block. But this permutation needs to be contained in $AGL_1(p)$, whence either $r=1$ and $d_1=p$, or $d_1=1$ and $d_2=\dots = d_r =:d$. The first scenario is impossible since $\sigma$ is assumed to contain only one $p$-cycle $(1,\dots, p)$ and all other cycles of length coprime to $p$. Setting $a:=\frac{p-1}{d}$, we are left with the case that $\sigma$ contains (among others) cycles of the following lengths:
\begin{itemize}
    \item[i)] One cycle of length $p$ for {\it some} prime $p\ge \log(n)^2$,
    \item[ii)] One cycle of length $m$, and $a$ cycles of length $\frac{m(p-1)}{a}$, for some $a|p-1$ and $m\in \mathbb{N}$.
    \end{itemize}

{\it Step IV}: 
We claim that the proportion of $\sigma\in S_n$ fulfilling i) and ii) above simultaneously is in $O(\frac{1}{\log(n)^2})$, which due to the above is enough to complete the proof. 

Indeed, due to Lemma \ref{lem:prob}, the proportion of $\sigma\in S_n$ fulfilling both i) and ii) is bounded from above by $$\sum\limits_{\stackrel{p\in \mathbb{P}}{p\ge \log(n)^2}}
\sum\limits_{a|p-1} \sum\limits_{m\in \mathbb{N}} \frac{1}{pm \cdot (\frac{m(p-1)}{a})^a \cdot a!} \le \sum\sum\sum \frac{a^a}{(p-1)^{a+1} m^{a+1}\cdot a!} = $$
$$ = \sum\limits_{\stackrel{p\in \mathbb{P}}{p\ge \log(n)^2}} \frac{1}{(p-1)^2}\sum\limits_{m\in \mathbb{N}} \frac{1}{m^2} + \sum\limits_{\stackrel{p\in \mathbb{P}}{p\ge \log(n)^2}}  \sum\limits_{2\le a|p-1}\sum_{m\in \mathbb{N}} \frac{a^a}{(p-1)^{a+1}m^{a+1} a!}$$
The first sum is obviously, up to constant factor $\le \frac{\pi^2}{6}$, bounded from above by $\sum\limits_{k=\log(n)^2}^\infty \frac{1}{k^2} = O(\frac{1}{\log(n)^2})$. To further bound the second sum, we once again use the inequality \\$a! \ge \sqrt{2\pi} a^{a+1/2}\exp(-a)$ to obtain an upper bound
$$\frac{1}{\sqrt{2\pi}} \sum\limits_{p\ge \log(n)^2}\sum_{2\le a|p-1}\sum\limits_{m} \frac{\exp(a)}{(p-1)^{a+1} m^3 \sqrt{a}}\le \frac{1}{\sqrt{2\pi}} \sum\limits_{m\ge 1} \frac{1}{m^3} \sum\limits_{p\ge \log(n)^2} \frac{\exp(2)}{(p-1)^3} \sum\limits_{a\le p} \frac{1}{\sqrt{a}}.$$
The last sum over $a$ is bounded from above by $2\sqrt{p}$, so the double sum over $p$ and $a$ is, up to constant factor, bounded from above by $\sum\limits_{p\ge \log(n)^2} \frac{1}{(p-1)^{5/2}} = O(\frac{1}{(\log(n)^2)^{3/2}})$. Since the sum over $m$ is absolutely bounded, the whole expression is in $O(\frac{1}{\log(n)^3})$, showing the claim and completing the proof.
\end{proof}

\begin{remark}
If one only intents to show the ``limit equals $1$" part of Theorem \ref{thm:main}, a few shortcuts in the above proof are possible. E.g., Step I may be shortened via replacing \cite{Unger} by \cite{GPU} which directly asserts that ``most" permutations $\sigma\in S_n$ power to a cycle of prime length $p\in [\log n, (\log n)^{\log \log n}]$; however, this comes at the cost of a worse asymptotic. Our own asymptotic bound should also still be far from optimal (compare the experimental data in the next section), mainly because the strong Condition \eqref{eq:2} is far from necessary to invariably generate a nonsolvable subgroup.
\end{remark}

The assertion of Theorem \ref{thm:main} can easily be translated into a statement about the mean value of the number of (independently chosen) random permutations required to invariably generate a nonsolvable subgroup of $S_n$.

\begin{cor}
\label{cor:mean}
Let $\sigma_1,\sigma_2,\dots$ be independent random permutations in $S_n$ or $A_n$, and let $N_n :=\min\{r \in \mathbb{N}\mid \sigma_1,\dots, \sigma_r \text{ invariably generate a nonsolvable group}\}$. Then $\lim_{n\to \infty} E(N_n) = 2$. 
\end{cor}
\begin{proof}
By Theorem \ref{thm:main}, there exists $\epsilon_n$, tending to $0$ with $n\to\infty$, such that $P(N_n=2)= 1-\epsilon_n$.
Furthermore, within the event $N_n>2$, the mean value is clearly bounded from above by $E(N_n)+2$ (simply by ignoring the first two permutations and seeking to invariably generate a nonsolvable group with the remaining ones). Therefore we have
$$E(N_n) \le 2\cdot(1-\epsilon_n) + (2+E(N_n))\cdot \epsilon_n,$$ or in other words $E(N_n) \le \dfrac{2}{1-\epsilon_n}$, which tends to $2$ for $n\to \infty$.
Conversely, $N_n$ cannot possibly be smaller than $2$, which shows the assertion.
\end{proof}


\section{Some computational data and further directions}
\label{sec:comp}

While Theorem \ref{thm:main} and Corollary \ref{cor:mean} give  assertions about the large degree limit, they say nothing about small values of $n$. We include some computational results on the number $N_n$ of random permutations required to invariably generate a nonsolvable subgroup of $S_n$, for some small $n$. Concretely, we calculated the probabilities $P(N_n=2)$. 
All computations were performed using Magma (\cite{Magma}).\footnote{Since the solvable subgroups of a group such as $S_{25}$ are too numerous to be enumerated in full, the required cycle structures of maximally solvable subgroups were generated iteratively, using the obvious fact that such a subgroup must either be transitive or a direct product of two maximally solvable subgroups of smaller degree.} The values are rounded to three decimal digits for convenience. 

\begin{table}[h!]
\begin{tabular}{c|c
|| c|c
||c|c
}
$G$ & $P(N_n=2)$ 
& $G$ & $P(N_n=2)$ & $G$ & $P(N_n=2)$ 
\\ \hline
$S_5$ & 0.250
& $S_{12}$ & 0.607
&  $S_{19}$ & 0.810 
\\
$S_6$ & 0.244
& $S_{13}$ & 0.660 
&  $S_{20}$ &0.821
\\
$S_7$ & 0.395
& $S_{14}$ & 0.700 
& $S_{21}$& 0.834 
\\
$S_8$ & 0.380
& $S_{15}$ & 0.723 
&  $S_{22}$& 0.851
\\
$S_9$ & 0.461
& $S_{16}$ & 0.730 
& $S_{23}$& 0.864
\\
$S_{10}$ & 0.543
&   $S_{17}$ & 0.764 
& $S_{24}$& 0.870
\\
$S_{11}$ & 0.601
& $S_{18}$& 0.788 
& $S_{25}$& 0.885
\end{tabular}
\caption{Values for small symmetric groups}
\label{tab:values}
\end{table}
\begin{figure}[h!]
    \centering
    \includegraphics[width=80mm]{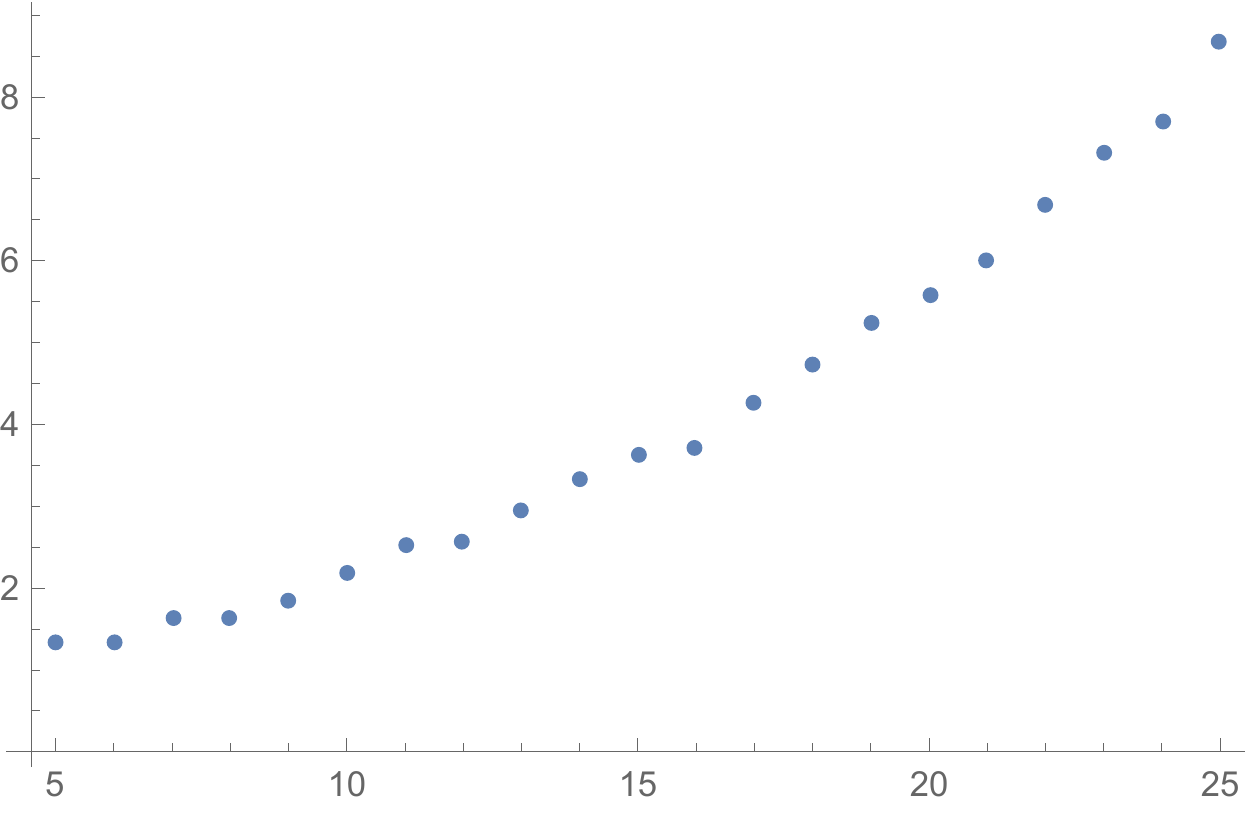}
    \caption{Plot of points $(n, (1-P(N_n=2))^{-1})$}
    \label{fig:plot}
\end{figure}


These values actually seem to suggest a superlinear convergence of the ``exceptional" probability $1-P(N_n=2)$ (see Figure \ref{fig:plot}), although we emphasize that the sample size is too small to draw definite conclusions. Note that $1/n^2$ is a trivial lower bound for $1-P(N_n=2)$, since that is the probability to draw two $n$-cycles, which do not even invariably generate a non-cyclic subgroup. A much deeper result by Blackburn, Britnell and Wildon (\cite[Theorem 1.6]{BBW}) states that the probability for two random permutations in $S_n$ to invariably generate a {\it non-abelian} subgroup is in fact still $1-O(1/n^2)$. It would be interesting to find out whether this (or a similar) asymptotic bound might indeed still hold for the case of invariable generation of non-solvable subgroups, although this should require a much more detailed analysis of cycle structures in solvable groups.

Finally, here are a few thoughts on the case of groups $G$ other than $A_n$ or $S_n$.
Obviously, a result such as Theorem \ref{thm:main} cannot be generalized to arbitrary families of nonsolvable groups $G$, and notably, the probability for two (or more) permutations of $G$ to invariably generate a nonsolvable subgroup of $S_n$ (in the sense of Definition \ref{def:inv}) depends on the given permutation action. E.g., if $G$ acts in its regular permutation action, then every cycle type of $G$ also occurs in the cyclic group of order $|G|$, i.e., for polynomials with such a Galois group, one can never obtain a nonsolvability certificate by relying solely on Dedekind's criterion. But even for primitive groups $G$ this can happen; e.g., $G=A_6$ in its (degree $15$) primitive action on $2$-sets contains only cycle structures which are also contained in the solvable group $S_3\wr AGL_1(5)$. For certain other groups, $P(N=2)>0$ is obvious, e.g., $PSL_2(p)\le S_{p+1}$ in its natural permutation action contains cycle structures $(p,1)$ and $(\frac{p+1}{2}, \frac{p+1}{2})$, which together invariably generate a doubly transitive and (for $p\ge 5$) nonsolvable subgroup of $S_{p+1}$. It would be interesting to know for which groups Dedekind's criterion can detect nonsolvability (i.e., $P(N=r)>0$ for some suitable $r$), resp., can detect it in two steps (i.e., $P(N=2)>0$).


\appendix
\section{A strengthening}
\label{sec:strong}
The proof of Theorem \ref{thm:main} can be adapted without many difficulties to show the following stronger result.

\begin{thm}
\label{thm:gen}
Let $G=S_n$ or $G=A_n$. Then the proportion of random permutations $\sigma,\tau\in G$ which invariably generate a subgroup of $S_n$ with at least one (nonsolvable) alternating composition factor tends to $1$ as $n\to \infty$.
\end{thm}

To show the strengthening, we need the classification of primitive permutation groups containing a cycle (e.g., \cite[Theorem 1.2]{Jones}).

\begin{prop}
\label{prop:cycles}
Let $G$ be a primitive group of degree $m>23$ not containing $A_m$. Assume that there exists an integer $\ell$ such that $G$ contains an $\ell$-cycle. Then one of the following holds:
\begin{itemize}
    \item[1)] $m=\ell$ prime, and $C_\ell\le G\le AGL_1(\ell)$,
    \item[2)] $m=\ell=(q^d-1)/(q-1)$ for some $d\ge 2$ and prime power $q$; and $PGL_d(q)\le G\le P\Gamma L_d(q)$,
    \item[3)] $m=\ell+1=q^d$ for some $d\ge 1$ and prime power $q$; 
    and $AGL_d(q)\le G\le A\Gamma L_d(q)$,
    \item[4)] $m=\ell+1$ with $\ell$ prime, and $PSL_2(\ell)\le G\le PGL_2(\ell)$; or
    \item[5)] $m=\ell+2=q+1$ for some prime power $q$; and $PGL_2(q)\le G\le P\Gamma L_2(q)$.
\end{itemize}
\end{prop}

We now show how to adapt the proof of Theorem \ref{thm:main} to obtain Theorem \ref{thm:gen}.
\begin{proof}[Proof of Theorem \ref{thm:gen}]
Choose a prime $p (\ge \log(n)^2$) as in Step I of the proof of Theorem \ref{thm:main}. I.e., the permutation $\sigma$ may be assumed to power to a $p$-cycle. Note next that in Cases 3) and 5) of Proposition \ref{prop:cycles} (with $\ell=p$), $p$ is necessarily a Mersenne prime, which is excluded in Step II of the proof. In Case 2) of Proposition \ref{prop:cycles}, if $d=2$, then $p$ is necessarily a Fermat prime, which can be excluded in the same way; on the other hand, the sum of all reciprocals of numbers $(q^d-1)/(q-1)$ with $d\ge 3$ and $q$ a prime power, is convergent (e.g., bounded from above by the some of all reciprocals of proper powers), hence for sufficiently large $n$, the sum of reciprocals of all such numbers which are additionally $\ge \log(n)^2$ becomes arbitrarily small. We may therefore alter Step II of the proof of Theorem \ref{thm:main} to additionally demand that $p$ is no such number. In particular, for the groups $U$ in Step III and $G$ in Step IV of the proof (which are primitive groups containing a $p$-cycle by construction), Proposition \ref{prop:cycles} then only leaves the possibility to be alternating or symmetric of degree $\ge p$, or to be contained in $AGL_1(p)$ (already dealt with) or in $PGL_2(p)$ in its primitive action on $p+1$ points. But the cycle structures in the latter are similarly restricted as in $AGL_1(p)$: they are all either $(p,1)$, or $(d_1,\dots, d_1)$, or $(d_2,\dots, d_2,1,1)$ with divisors $d_1$ of $p+1$ and $d_2$ of $p-1$. Now using these cycle structures, Step III of the proof may be carried out essentially as above, to show that with probability tending to $1$, $U$ must be alternating, symmetric or imprimitive. Finally, it remains to adapt Step IV to bound the probability for $\sigma$ to have both a $p$-cycle for some prime $p\ge \log(n)^2$, and $\frac{p+1}{d_1}$ cycles of the same length $md_1$ for some divisor $d_1$ of $p+1$ and some $m\in \mathbb{N}$.\footnote{This corresponds to the cycle type $(d_1,\dots, d_1)$ of $PGL_2(p)$ above; the estimate for the cycle types $(d_2,\dots, d_2,1,1)$ is analogous, and indeed easier.} This yields an upper bound essentially as in Step IV above, except that we lose one factor $\frac{1}{m}$. The reader may verify directly that the only somewhat ``critical" case is then $d_1=p+1$ (i.e., $\sigma$ has a $p$-cycle as well as one single cycle of length divisible by $p+1$), which gives a contribution of at most $$\sum\limits_{\stackrel{p\in \mathbb{P}}{p\ge \log(n)^2}} \frac{1}{p(p+1)} \sum_{m=1}^n\frac{1}{m} \in O(\frac{1}{\log(n)^2} \cdot \log(n)) = O(1/\log(n)),$$
which converges to $0$. 
It therefore has been shown that with probability tending to $1$, $\sigma$ and $\tau$ invariably generate a subgroup with at least one alternating composition factor: namely either coming from a homomorphic image $U$ acting primitively as the alternating or symmetric group on some orbit; or from an imprimitive action of a homomorphic image $U$ in which the block kernel projects onto an alternating or symmetric group.
%
%
\end{proof}

\begin{remark} In the above proof, we have made no effort to retain the asymptotic of Theorem \ref{thm:main}, although this could be achieved; e.g., in the last step, since the only critical case is the one where there exists only a single cycle of length $m(p+1)$, the imprimitive group $U$ must have a total number of $m+1$ blocks of length $p+1$ (one containing the $p$-cycle, and $m$ further ones permuted cyclically by $\sigma$). The image of $U$ in the blocks action is then a transitive group of degree $m+1$ containing an $m$-cycle, hence even $2$-transitive and thus covered by Proposition \ref{prop:cycles}. Using this, the set of admissible values $m$ can be restricted sufficiently. \end{remark}

\end{document}